\documentclass[a4paper,10pt]{amsart}
\usepackage[utf8]{inputenc}
\usepackage{amsmath,amssymb,amsthm,amscd,amsfonts,mathrsfs,verbatim,hyperref,stmaryrd}
\usepackage{orcidlink}
\usepackage[all]{xy}

\newtheorem{theorem}{Theorem}[section]
\newtheorem{remark}[theorem]{Remark}

\newtheorem{proposition}[theorem]{Proposition}

\newtheorem{lemma}[theorem]{Lemma}

\def\N{\mathbb N}
\def\Z{\mathbb Z}

\def\C{\mathbb C}
\def\d{\operatorname{d}}
\newcommand{\gl}{{\mathfrak{gl}}}
\renewcommand{\b}{{\mathfrak{b}}}
\newcommand{\n}{{\mathfrak{n}}}

\newcommand{\e}{\mathbf e}
\newcommand{\f}{\mathbf f}
\renewcommand{\v}{\mathbf v}
\newcommand{\w}{\mathbf w}
\renewcommand{\O}{{\mathcal{O}}}

\title{Moduli of atoms of complex projective varieties}
\author{Maxim Kontsevich, Szil\'ard Szab\'o}
\address{Maxim Kontsevich: Institut des Hautes Etudes Scientifiques, Le Bois-Marie, 35 route de Chartres, 91893 Bures-sur-Yvette, France \newline  Szil\'ard Szab\'o: Institute of Mathematics,
Faculty of Science, E\"otv\"os Lor\'and University, P\'azm\'any P\'eter s\'et\'any 1/C, Budapest, Hungary, H-1117, \texttt{szilard.szabo@ttk.elte.hu}; HUN-REN Alfr\'ed R\'enyi Institute of Mathematics, Re\'altanoda utca 13-15., Budapest 1053, Hungary, \texttt{szabo.szilard@renyi.hu}}

\begin{document}

\begin{abstract}
The category of constituent pieces of Dubrovin connection of varieties is introduced and shown to be representable in affine space.
\end{abstract}

\maketitle

\section{Introduction}\label{sec:intro}

Let us fix an integer $r\geq 2$ and a constant $c_i\in\C^{\times} \setminus \{ 1 \}$ for each $1\leq i \leq r$.
Set $\vec{c} = (c_1, \ldots , c_r)$.  
We will denote by $B\subset \operatorname{GL} (r, \C )$ the standard Borel subgroup consisting of invertible upper triangular matrices, and by $U$ its unipotent radical. 
We denote the identity matrix of dimension $r$ by $\mbox{I}_r$. 
We let $H\subset \operatorname{GL}(r,\C )$ stand for the maximal torus of $B$, so $H\cong B/U \cong (\C^{\times})^r$. 
We let $\b\subset \gl_r (\C )$ stand for the Lie algebra of $B$ and $\n$ for its nilpotent radical, i.e. the subspace of strictly upper triangular matrices.
By the flag defined by a basis $(\f_1, \ldots , \f_r)$ of a vector space $W$ we mean 
\[
0 \subset \C \f_1 \subset  \C \f_1 \oplus  \C \f_2 \subset \cdots \subset \C \f_1 \oplus \cdots \oplus \C \f_r = W. 
\]

Fix some $\varepsilon>0$ and define 
\[
 \Delta = \{ z\in\C \colon \quad |z| < 1+\varepsilon \} .
\]
Let $\O$ stand for the sheaf of holomorphic functions on $\Delta$ and $\O^{\times}$ the subgroup of invertible elements\footnote{More generally, one can let $\varepsilon>0$ vary, and consider the injective limit of the rings of holomorphic functions on discs of radius $>1$.}.
Let $\Omega$ denote the sheaf of holomorphic $1$-forms on $\Delta$ (i.e., the free $\O$-module of rank $1$ generated by the symbol $\d\! z$) and $\Omega (*\{ 0 \}) = \Omega \otimes_{\O} \O (*\{ 0 \})$ stand for the sheaf of meromorphic $1$-forms having a pole of arbitrary order.
Our aim is to introduce and study a groupoid $NSQC_{\vec{c}}$ whose objects are triples $(E, \nabla,  (\v_1 (1), \ldots ,\v_r (1)))$, where 
\begin{enumerate}
 \item[(Ob1)] $E$ is a free $\O$-module of rank $r$,
 \item[(Ob2)] $\nabla\colon E\to E\otimes \Omega (*\{ 0 \})$ is a meromorphic connection, i.e. a $\C$-linear operator obeying Leibniz' rule~\eqref{eq:Leibniz},
 \item[(Ob3)] a basis $(\v_1 (1), \ldots ,\v_r (1))$ of $E_{|1}$,
\end{enumerate}
subject to the conditions 
\begin{enumerate}
 \item[(Cond1)] $\nabla$ has regular singularity at $0$ \label{item:regular}
 \item[(Cond2)] writing $K \d\! z$ for the connection matrix of $\nabla$ with respect to the standard basis of $E$, $K$ has a pole of order at most $2$ at $0$, 
 \item[(Cond3)] letting $T_{\nabla}$ denote the monodromy transformation of $\nabla$ along a loop based at $1\in\C$ and winding around $0$ once in positive direction, there exist $t_{ji} \in \C$ with 
 \[
  T_{\nabla} (\v_i(1)) =  c_i \v_i(1) + \sum_{j<i} t_{ji} \v_j(1)
 \]
 (i.e., the flag of $E_{|1}$ determined by $(\v_1 (1), \ldots ,\v_r (1))$ is preserved by $T_{\nabla}$, with generalized eigenvalue  associated to $\v_i (1)$ equal to $c_i$).
\end{enumerate}

We denote by $\operatorname{Aut}(E)$ the group 
\[
 g\colon  \Delta \to \operatorname{GL} (r, \C )
\]
for pointwise multiplication. 
The morphisms of $NSQC_{\vec{c}}$ are given by $g\in \operatorname{Aut}(E)$.
The gauge action on $(E, \nabla )$ is the standard one, spelled out in~\eqref{eq:gauge_action}. 
The action of $g$ on each element of the basis $\v_i (1)$ is by the standard representation of $g(1)$.
In particular, a stabilizer $h$ of any object must satisfy $h(1) = \mbox{I}_r$, and  $h$ must be constant too by $h \circ \nabla = \nabla \circ h$.
Said differently, $\operatorname{Aut}(E)$ acts freely on the objects of $NSQC_{\vec{c}}$.

We now state our main result (see Theorem~\ref{thm:main}). 
\begin{theorem}
 Assume that $c_i \neq c_j$ for all $1\leq j \neq i \leq r$. 
 The connected components of the coarse moduli space of the groupoid $NSQC_{\vec{c}}$ are isomorphic to $\n \oplus \n$ (and thus to an affine space of dimension $r(r-1)$ over $\C$). 
\end{theorem}

This will follow from Proposition~\ref{prop:main}, where we reduce formal meromorphic connections (possibly with pole orders higher than $2$) to a normal form that is analogous to the Levelt normal form~\cite[Exercise~2.20]{Sabbah_Isomonodromic_deformations} of logarithmic connections. 

Let us now comment on the reason why the conditions on the objects of $NSQC_{\vec{c}}$ are quite natural. 
Many interesting examples of quantum connections arise in Mirror Symmetry as Gauss--Manin systems of Landau--Ginzburg models $f\colon X\to \mathbb{C}$~\cite[Section~VII.5]{Sabbah_Isomonodromic_deformations}. 
Since these examples are obtained by taking partial Fourier--Laplace transform of the de Rham complex of $X$ twisted by $f$, it follows from the stationary phase formula that they are isomorphic over $\C \llbracket z \rrbracket$ to direct sums of meromorphic connections of the form 
\begin{equation}\label{eq:decomposition}
 \bigoplus_{a\in\mathbb{C}} \left( \d + \frac{a}{z^2} \d\! z \right) \otimes \nabla_{\operatorname{reg}}^a, 
\end{equation}
where the summation is finite and $\nabla_{\operatorname{reg}}^a$ has regular singularity at $z=0$. 
The set of numbers $a$ appearing non-trivially in this decomposition is called the spectrum of the connection. 
In the case of Fourier--Laplace transform of the twisted de Rham complex of $X$, the spectrum agrees with the set of critical values of $f$ (up to a sign). 
In other words, in such examples base change to a root of $z$ in the Turrittin--Hukuhara--Levelt decomposition is not needed. 
This class of examples is called non-commutative Hodge structures of exponential type~\cite[Definition~2.12]{Katzarkov_Kontsevich_Pantev}. 
Let $K_j$ denote the degree $j$ Laurent coefficient of $K$ at $0$. 
The restriction of $K_{-2}$ to any generalized eigenspace is the sum of a simple (i.e. constant) endomorphism and a nilpotent one. 
Up to simple endomorphisms, the study can therefore be reduced to the case when $K_{-2}$ is equal to a nilpotent endomorphism, and the connection has regular singularity. 
This explains assumption~(Cond1). 
(We will see in Proposition~\ref{prop:main}~\eqref{item:nilpotent},~\eqref{item:conjugate} that~(Cond1) and~(Cond2) imply that $K_{-2}$ must be nilpotent; if, moreover, $K_{-2}\in\n$ is a regular nilpotent element, then this implies $K_{-1}\in\b$ (Proposition~\ref{prop:nilpotent})).
Finally, we add the flag to the data in order to rigidify the functor (remove possible automorphisms). 
Under the non-resonance condition $c_i \neq c_j$, up to applying an element of the standard Borel subgroup, it simply amounts to a choice of the eigenvalues of the monodromy.

Let us now describe the context of this study. 
Since the 1990's, the study of Frobenius manifolds~\cite{Dubrovin_2D} has played a central role in Mathematical Physics, with close ties to Singularity Theory, Enumerative Geometry, Isomonodromic Deformations and other mathematical areas. 
As a matter of fact, in this article we will adopt the setup of (formal) $F$-manifolds~\cite{Hertling_Manin}, a slightly weakened version of Frobenius manifolds for which a flat metric is not required to exist. 
An $F$-manifold is called semisimple when the spectrum of quantum multiplication by the Euler vector field $Eu*$ is a semisimple endomorphism. 
As it is well-known, this operation is the degree $-2$ coefficient of the connection form of the quantum connection with respect to a natural local chart and trivialization. 
A considerable amount of literature has addressed the semisimple case. 
However, due to a result of Bayer and Manin~\cite{Bayer_Manin} later  strengthened by  Hertling, Manin and Teleman~\cite{Hertling_Manin_Teleman}, the quantum cohomology of a smooth projective variety can only be semisimple if it has no odd cohomology and is of Hodge--Tate type. 
Said differently, the quantum cohomology of many interesting varieties (e.g. Fano $3$-folds) is not semisimple. 
It is therefore interesting to extend the theory beyond the semisimple case. 
The aim of this paper is to initiate the study of non-semisimple quantum connections by establishing a particularly simple normal form for them (Theorem~\ref{thm:main}), and defining their moduli space. 

More recent motivation is provided by the proof of Kontsevich' blow-up conjecture by H.~Iritani~\cite{Iritani} on the behavior of the spectrum of the Dubrovin (also known as quantum) connection with respect to blow-ups, and its application to birational geometry~\cite{Katzarkov_Kontsevich_Pantev_Yu}.
The idea of this application consists in introducing a notion of ``atom'', denoted $\operatorname{Atom}(X)$, associated to a smooth projective variety $X$. 
Roughly speaking, $\operatorname{Atom}(X)$ is the set of direct summands of the   Dubrovin connection of $X$ appearing in the decomposition~\eqref{eq:decomposition} counted with their multiplicity, up to the following equivalence relation: 
\begin{enumerate}
 \item $\operatorname{Atom}(X_1 \sqcup X_2) \sim \operatorname{Atom}(X_1) +  \operatorname{Atom}(X_2)$; 
 \item for any blow-up 
  \[
  \widetilde{X} \to X 
  \]
  with smooth center $Z$ of codimension $r \geq 2$, 
  \[
  \operatorname{Atom}(\widetilde{X}) \sim \tau^* \operatorname{Atom}(X) + \left( \sum_{i=1}^{r-1} \varsigma_j^* \operatorname{Atom}(Z) \right)
  \]
  for suitable formal changes of quantum variables $\tau, \varsigma_j$;
  \item (quantum Leray--Hirsch) \cite[Theorem~0.2]{LLW} for a vector bundle $V\rightarrow B$ over a projective base, a certain relation between the class of the fiberwise projectivization $\operatorname{Proj}_B(V)$ and the class of $B$. 
\end{enumerate}
The precise nature of this equivalence relation is yet to be explored.
Once this is done, the atom of $X$ (together with the point spectrum $\{ c \}$ of its Dubrovin connection) will provide a fine birational invariant of $X$. 
By virtue of the Weak Factorization Theorem~\cite[Theorem~0.3.1.]{AKMW}, we then get a birationally invariant map 
\begin{align*}
  \operatorname{Atom}\colon \operatorname{ProjVar}_{\mathbb{C}} & \to \mathbb{N}^{\oplus_{\vec{c}} NSQC_{\vec{c}}}/\sim \\
  X & \mapsto \operatorname{Atom}(X),
\end{align*}
where $\operatorname{ProjVar}_{\mathbb{C}}$ stands for the set of projective varieties over $\mathbb{C}$, $NSQC_{\vec{c}}$ for the coarse moduli space of the functor denoted the same way, and the right hand side stands for finite linear combinations of elements of $NSQC_{\vec{c}}$ for all $\vec{c}$. 
For the sake of completeness, let us mention that the blow-up conjecture states that the spectrum of the Dubrovin connection of  $\widetilde{X}$ for small quantum parameter $q$, is asymptotically equal to that of $X$, union the points $\zeta^j q^{-1} \quad (0\leq j \leq r-2)$, 
where $\zeta$ is a primitive $(r-1)$'th root of unity. 
For more details about the case $\operatorname{dim}_{\mathbb{C}}(X) = 2 = r$, see also~\cite{Gyenge_Szabo}. 

\subsection*{Acknowledgments}
The authors are grateful to the referee for a thorough reading and useful suggestions. 
The second author would like to thank Claude Sabbah for  inspiring discussions and explanations, and the hospitality of the Institut des Hautes Etudes Scientifiques (Bures-sur-Yvette). 
The second author received funding from the grants NKFIH KKP 144148 and NKFIH K  	146401. 

\section{Meromorphic connections, lattices, regular singularity}

Here we collect some well-known material, see for instance~\cite{Sabbah_Isomonodromic_deformations}. 

Fix an integer $r\geq 2$. 
We denote by $V$ the free $\C \llbracket z \rrbracket$-module $\C \llbracket z \rrbracket^{\oplus r}$ of rank $r$, with its standard basis $\e_1, \ldots ,\e_r$. 
Consider the $\C$-linear operator defined by 
\begin{align*}
 \d\colon \C \llbracket z \rrbracket & \to \Omega \\
 z^n & \mapsto n z^{n-1} \d\! z. 
\end{align*}
A formal connection in $V$ is a $\C$-linear map: 
\[
 \nabla\colon V\to V\otimes_{\C \llbracket z \rrbracket} \Omega
\]
satisfying Leibniz' rule 
\begin{equation}\label{eq:Leibniz}
     \nabla (f \otimes \sigma ) = (\d\! f) \otimes \sigma + f \otimes \nabla ( \sigma ) 
\end{equation}
for any $f\in \C \llbracket z \rrbracket, \sigma\in \C \llbracket z \rrbracket^{\oplus r}$. 
Consider now the field of fractions 
\[
 \C (\!(z)\!) = \C \llbracket z \rrbracket [z^{-1}] 
\]
of $\C \llbracket z \rrbracket$. 
We define the $z$-adic valuation by 
\begin{align*}
 \operatorname{val}_z \colon \C (\!(z)\!) & \to \Z \\
 f = \sum_{n=-N}^{\infty} f_n z^n & \mapsto -N 
\end{align*}
whenever $f_{-N}\neq 0$. 
A formal meromorphic connection in $V$ is a $\C$-linear map: 
\[
 \nabla\colon V\otimes_{\C \llbracket z \rrbracket}\C (\!(z)\!)\to V\otimes_{\C \llbracket z \rrbracket} \C (\!(z)\!) \otimes_{\C \llbracket z \rrbracket} \Omega 
\]
satisfying Leibniz' rule as above, this time for any $f\in \C \llbracket z \rrbracket [z^{-1}]$.  

Let $\nabla$ be a formal meromorphic connection in $V$. 
With respect to $\e_1, \ldots \e_r$, $\nabla$ then admits the connection matrix form 
\begin{equation}\label{eq:connection_series}
    \nabla = \d + K \d\! z  = \d + \sum_{n=-N}^{\infty} K_n z^n \d\! z   
\end{equation}
for some integer $N>0$ and $K\in  \C (\!(z)\!)\otimes_{\C} \gl_r (\C )$, respectively $K_n \in \gl_r (\C )$. 
Namely, for any $f_1, \ldots , f_r \in \C (\!(z)\!)$ we have 
\begin{equation}\label{eq:connection_matrix}
 \nabla (f_1 \e_1 + \cdots + f_r \e_r ) = (\e_1, \ldots ,\e_r) ( \d + K\d\! z ) \begin{pmatrix}
    f_1 \\ \vdots \\ f_r
    \end{pmatrix},
\end{equation}
$\d$ acting on column vectors entrywise. 
We define the $z$-adic valuation of an element $K\in \C (\!(z)\!)\otimes_{\C} \gl_r (\C )$ as the minimum of the valuations of its entries, i.e. the smallest value of $-N$ such that $K_{-N}\neq 0$ in the above expansion. 
(The case $N=2$ is particularly relevant for applications to the functors $NSQC$, but the results of Section~\ref{sec:normal} hold for general $N>1$.) 
A lattice in $V\otimes \C (\!(z)\!)$ is a $\C \llbracket z \rrbracket$-submodule $V'$ that generates $V\otimes \C (\!(z)\!)$ as a $\C (\!(z)\!)$-vector space. 
The connection matrix of $\nabla$ with respect to a $\C \llbracket z \rrbracket$-basis of $V'$ is defined analogously to~\eqref{eq:connection_matrix}. 
A cyclic vector of $(V,\nabla )$ is $\e\in V\otimes \C (\!(z)\!)$ such that $\e, \nabla \e, \ldots , \nabla^{r-1} \e $ span a lattice in $V\otimes \C (\!(z)\!)$. 
We say that $\nabla$ admits a regular singularity at $\operatorname{Specf} \C \llbracket z \rrbracket/(z)$ if $V\otimes \C (\!(z)\!)$ admits a lattice $V'$ such that the connection matrix of $\nabla$ has only first-order poles on $V'$. 

The above notions can be defined for $\C \llbracket z \rrbracket$ replaced by $\O$, leading to the notion of holomorphic/meromorphic connections. 
Clearly, tensor product by $\C \llbracket z \rrbracket$ over $\O$ gives rise to a functor from the category of holomorphic/meromorphic connections to the category of formal holomorphic/meromorphic connections. 
In this context, \cite[Th\'eor\`eme~(5.3).i)]{Malgrange} states that the restriction of this functor to the full subcategory of objects with regular singularity is an equivalence. 

The space of formal meromorphic connections in $V$ is an affine space modelled on $\C (\!(z)\!) \otimes_{\C} \gl_r (\C )$, and the group $\operatorname{GL}(r, \C (\!(z)\!))$ acts on it: for any $g\in \operatorname{GL}(r, \C (\!(z)\!))$, 
\begin{equation}\label{eq:gauge_action}
    (\d\! + K \d\! z) \cdot g = \d\! + g^{-1} K g \d\! z + g^{-1}\d\! g. 
\end{equation}

\begin{proposition}\label{prop:nilpotent}
 Let $\nabla$ be a meromorphic connection on $V$ with regular singularity at $0$.
 Let $K$ denote the connection matrix of $\nabla$ with respect to the trivialization  $\e_1, \ldots ,\e_r$.
 Assume that 
 \begin{enumerate}
  \item $K$ has a pole of order $2$, 
  \item $K_{-2}\in\n$, 
  \item $K_{-2}$ is a regular endomorphism (i.e., its centralizer in $\operatorname{GL}(r, \mathbb{C})$ is of the lowest possible dimension $r$). 
 \end{enumerate}
 Then we must have $K_{-1}\in\b$. 
\end{proposition}

\begin{remark}
 As it is known (and will actually follow from Proposition~\ref{prop:main}~\eqref{item:nilpotent},~\eqref{item:conjugate} too), regularity implies that $K_{-2}$ is a nilpotent endomorphism. 
 Thus, assuming the first condition, the second condition holds up to changing the trivialization by a constant automorphism. 
\end{remark}

Our proof will consist in reducing differential modules to the higher order scalar case by finding a cyclic vector. 
We will make use of Fuchs' classical regularity criterion: let us be given a scalar differential operator of rank $r$ defined for a variable $y\in \mathbb{C} \llbracket z \rrbracket$, 
\[
 L (y) = y^{(r)} + a_1 (z) y^{(r-1)} + \cdots + a_r (z) y, 
\]
where $y^{(i)}$ stands for the $i$'th derivative with respect to $z$, and $a_i\in \mathbb{C} \llbracket z \rrbracket$ are fixed. 
Then $L$ has regular singularity at $0$ if and only if 
\begin{equation}\label{eq:Fuchs}
  \operatorname{val}_z (a_i) \geq - i 
\end{equation}
for all $1 \leq i \leq r$.

\begin{proof}
 We first deal with the case $r=2$. 
 First applying a constant change of basis $g_0\in \operatorname{GL}_r (\C )$ we may assume that 
 \[
  K_{-2} = \begin{pmatrix}
            0 & 1 \\
            0 & 0
           \end{pmatrix}. 
 \]
 Let us now write 
 \[
  K_{-1} = \begin{pmatrix}
            \alpha_{-1} & \beta_{-1} \\
            \gamma_{-1} & \delta_{-1}
           \end{pmatrix}. 
 \]
 \begin{lemma}\label{lem:gamma0}
  We have $\gamma_{-1} = 0$. 
 \end{lemma}
 \begin{proof}[Proof of the Lemma]
  For simplicity, we write $\nabla$ for $\nabla_{\partial_z}$. 
  We have 
  \begin{align*}
   \nabla \e_2 & = (z^{-2} + O(z^{-1}) )\e_1 + O(z^{-1}) \e_2 \\ 
   \nabla \e_1 & = (\gamma_{-1} z^{-1} + O(1)) \e_2 + O(z^{-1}) \e_1 \\
   & = (\gamma_{-1} z^{-1} + O(1)) \e_2 + O(z) \nabla \e_2 .
  \end{align*}
  In particular, we find 
  \begin{align*}
   \nabla^2 \e_2 & = (-2 z^{-3} + O (z^{-2})) \e_1 + (z^{-2} + O(z^{-1}) ) ((\gamma_{-1} z^{-1} + O(1)) \e_2 + O(z) \nabla \e_2) \\
   & = (\gamma_{-1} z^{-3} + f(z)) \e_2 + g(z) \nabla \e_2
  \end{align*}
  for some 
  \[
  f(z) = O(z^{-2}), \quad g(z) = O(z^{-1}). 
  \]
  Now we choose the cyclic vector $\e_2$, i.e. we consider the trivialization $(\e_2, \nabla \e_2)$. 
  For any local section 
  \[
   \vec{y}(z) = y_0 (z) \e_2 + y_1 (z) \nabla \e_2 ,
  \]
  let us compute the system of equations expressing  $\nabla \vec{y} = 0$. 
  The coefficients of $\nabla \e_2$ and $\e_2$ give respectively 
  \begin{align}
      y_1' + g(z) y_1 + y_0 & = 0 \label{eq:parallel1} \\
      y_0' + y_1 (\gamma_{-1} z^{-3} + f(z)) & = 0. \label{eq:parallel2}
  \end{align}
  Plugging the differential of~\eqref{eq:parallel2} into~\eqref{eq:parallel1} yields 
  \[
   - y_1'' - g(z) y_1' + (\gamma_{-1} z^{-3} + f(z) - g'(z)) y_1 = 0. 
  \]
  So, we have reduced $\nabla$ to a second order scalar equation. 
  Now, given that $g(z) = O(z^{-1})$ we have $g'(z) = O(z^{-2})$ too. 
  By~\eqref{eq:Fuchs}, regularity is equivalent to 
  \[
    \operatorname{val}_z (g) \geq -1, \quad \operatorname{val}_z(\gamma_{-1} z^{-3} + O(z^{-2}))\geq -2. 
  \]
 We are given the first of these conditions by assumption. 
 The second one is equivalent to $\gamma_{-1} = 0$. 
 \end{proof}
 Let us now come to the case $r\geq 3$. 
 We may again assume that $K_{-2}$ is an upper triangular Jordan block of dimension $r$ for the  eigenvalue $0$. 
 We will just indicate the main steps of the proof, and leave the reader to work out the details along the lines of the rank $2$ case. 
 Let us denote the coefficients of $K_{-1}$ by $k_{ji}$. 
 We need to show that $k_{ji}=0$ for all $1 \leq i<j\leq r$. 
 It is easy to see that the vectors $\e_r, \nabla \e_r, \ldots, \nabla^{r-1} \e_r$ span $V$ over $\mathbb{C} (\! ( z )\!) $, so we may choose $\e_r$ as cyclic vector.
 After some computation, we find the expansion 
 \begin{align*}
  \nabla^r \e_r = & O(z^{-1}) \nabla^{r-1} \e_r + (k_{21} z^{-3} + O(z^{-2})) \nabla^{r-2} \e_r + \cdots  \\
  & \cdots + (k_{r1} z^{1-2r} + \cdots + k_{r,r-1} z^{-1-r} +  O(z^{-r})) \e_r,
 \end{align*}
 i.e. the first $i-1$ Laurent coefficients of the expansion of $\nabla^{r-i} \e_r$ are the elements in the $i$'th row of $K_{-1}$ strictly below the diagonal. 
 Fuchs' criterion~\eqref{eq:Fuchs} and the regularity assumption then give the desired vanishing statements. 
\end{proof}

\section{Normal form}\label{sec:normal}

Let $E = \mathcal{O}^{\oplus r}$ and $\nabla$ be a connection on $E$ with regular singularity at $0$.
Let $\v_1 (1), \ldots ,\v_r (1)$ be a frame of $E_{|1}$ whose flag is preserved by the matrix of the monodromy transformation $T_{\nabla}$ of $\nabla$, see assumption~(Cond3).
Let us write~\eqref{eq:connection_series} for the Laurent expansion of the  connection matrix of $\nabla$ with respect to $\e_1, \ldots ,\e_r$.
 Let $A\in \gl (E_{|1})$ be such that
 \begin{equation}\label{eq:expA}
    \exp (2\pi\sqrt{-1} A ) = T_{\nabla }.
 \end{equation}
 By assumption~(Cond3), the matrix of $A$ with respect to $(\v_1 (1), \ldots ,\v_r (1))$ is upper triangular.
 By assumption~(Cond1), it follows from~\cite{Deligne} that there exists a $\C (\!(z)\!)$-basis $(\v'_1, \v'_2 \ldots , \v'_r)$ of $E\otimes \C (\!(z)\!)$ with respect to which the connection matrix of $\nabla$ is $A\d\! z/z$ and such that
 \begin{equation}\label{eq:Id@1}
     \v'_i (1) = \v_i (1).
  \end{equation}
 The lattice $E'$ spanned by the frame $(\v'_1, \v'_2 \ldots , \v'_r)$ is in general different from $E$, but the meromorphic bundles agree:
 \[
  \C (\!(z)\!) \e_1 \oplus \cdots \oplus \C (\!(z)\!) \e_r = E\otimes \C (\!(z)\!) = \C (\!(z)\!) \v'_1 \oplus \cdots \oplus \C (\!(z)\!) \v'_r.
 \]
 It follows that there exists some $n \in \mathbb{Z}$ such that $z^n \v'_1$ belongs to the original lattice $E$ at $0$.
 We now let $n_1$ stand for the smallest such integer $n$, and set $\w_1 = z^{n_1}\v'_1\in E$.
 Then, the matrix of $\nabla$ with respect to $(\w_1, \v'_2 \ldots , \v'_r)$ reads as $(A' + n_1 E_{11})\d\! z/z$, where $E_{ij}$ has $(i,j)$-entry $1$ and all other entries $0$.
 In particular, this matrix is still upper triangular.
 Now, we recursively define
 \[
  n_i = \min \{ n \in \mathbb{Z} \colon \;  z^n \v'_i\in E + \C (\!(z)\!) \w_{i-1} + \cdots + \C (\!(z)\!) \w_1\}.
 \]
 \begin{proposition}\label{lem:frame}
 There exist unique $b_{1i},\ldots , b_{i-1,i}\in z^{-1} \C [ z^{-1} ]$ and unique $\w_i \in E$ satisfying the equality
  \begin{equation}\label{eq:vi}
  z^{n_i} \v'_i = \w_i -  b_{i-1,i} \w_{i-1}  - \cdots -  b_{1i} \w_1.
 \end{equation}
 In addition, the span of $\w_1, \ldots, \w_r$ over $\mathcal{O}$ agrees with $E$ .
 \end{proposition}
 \begin{proof}
 We apply induction. The base case $i=1$ has been explained above.
 So, assume that $\w_1, \ldots, \w_{i-1}$ with the given properties exist.
 By definition, there exist $\mathbf{u}_i\in E$ and $c_{1i}, \ldots , c_{i-1,i} \in \C (\!(z)\!)$ such that
 \[
  z^{n_i} \v'_i = \mathbf{u}_i - c_{i-1,i} \w_{i-1}  - \cdots - c_{1i} \w_1.
 \]
 Now, for all $j<i$ let us take the unique $b_{ji}\in z^{-1} \C [ z^{-1} ]$ such that
 \[
    c_{ji} - b_{ji} \in \C \llbracket z \rrbracket .
 \]
 By this and $\w_1, \ldots, \w_{i-1} \in E$, it follows that
 \[
  c_{i-1,i} \w_{i-1} + \cdots + c_{1i} \w_1  - (b_{i-1,i} \w_{i-1} + \cdots + b_{1i} \w_1 ) \in E
 \]
 because $E$ is a $\C \llbracket z \rrbracket$-module.
 Let us now set
  \[
  \w_i = \mathbf{u}_i + b_{i-1,i} \w_{i-1} + \cdots +  b_{1i} \w_1 - (c_{i-1,i} \w_{i-1} + \cdots + c_{1i} \w_1 ) \in E.
 \]
 Rearranging terms, we then get~\eqref{eq:vi}. This shows existence.

 Now, let
 \begin{align*}
  z^{n_i} \v'_i & = \w_i - b_{i-1,i} \w_{i-1} - \cdots -  b_{1i} \w_1  \\
  & = \tilde{\w}_i - \tilde{b}_{i-1,i} \w_{i-1} - \cdots - \tilde{b}_{1i} \w_1
 \end{align*}
 be another such decomposition for some $\tilde{\w}_i \in E$ and $\tilde{b}_{i-1,i} ,\ldots , \tilde{b}_{1i} \in z^{-1} \C [ z^{-1} ]$.
 Rearranging terms, we find
 \[
  \w_i - \tilde{\w}_i = (b_{i-1,i} -\tilde{b}_{i-1,i}) \w_{i-1} + \cdots + (b_{1i} - \tilde{b}_{1i}) \w_1 .
 \]
 Assume that at least one of the coefficients $b_{ji} - \tilde{b}_{ji}$ on the right hand side does not vanish, and let us fix the highest value of $j$ such that this is the case.
 We may then rewrite the above equality as
 \[
  (b_{j, i} - \tilde{b}_{j, i}) \w_j  = \w_i - \tilde{\w}_i - \sum_{k=1}^{j-1} (b_{k, i} - \tilde{b}_{k, i}) \w_k .
 \]
 Plugging
 \[
   \w_j = z^{n_j} \v'_j +  b_{j-1,j} \w_{j-1} + \cdots + b_{1j} \w_1
 \]
 to the left hand side, we get
 \[
  (b_{j, i} - \tilde{b}_{j, i}) z^{n_j} \v'_j = \w_i - \tilde{\w}_i - \sum_{k=1}^{j-1} d_{k, i} \w_k ,
 \]
 where we have set
 \[
  d_{k, i} = - b_{k, i} + \tilde{b}_{k, i} - (b_{j, i} - \tilde{b}_{j, i}) b_{kj}.
 \]
 Given that $b_{k, i} , \tilde{b}_{k, i} \in z^{-1} \C [ z^{-1} ]$, we also have $d_{k, i} \in z^{-1} \C [ z^{-1} ]$.
 Since $b_{j, i} - \tilde{b}_{j, i} \in z^{-1} \C [ z^{-1} ]$ is non-vanishing,  we may uniquely write it in the form
 \[
  b_{j, i} - \tilde{b}_{j, i} = z^{-m_j} P(z)
 \]
 with $m_j \in \mathbb{Z}_+$ and $P\in \C [z]$ such that $P(0)\neq 0$.
 Then, we can rewrite
 \[
  z^{n_j - m_j} \v'_j = P^{-1}(z) ( \w_i - \tilde{\w}_i )  - \sum_{k=1}^{j-1} P^{-1}(z) d_{k, i} \w_k .
 \]
 Now, we have $\w_i - \tilde{\w}_i \in V$ and $P^{-1}\in \C \llbracket z \rrbracket$.
 Since $E\otimes \llbracket z \rrbracket$ is an $\C \llbracket z \rrbracket$-module, we have $P^{-1}(z) ( \w_i - \tilde{\w}_i ) \in E$, contradicting the minimal choice of $n_j$.
 We conclude that none of the coefficients $b_{ji} - \tilde{b}_{ji}$ may be non-trivial.
 This finishes the proof of uniqueness.

 Finally, let us show that the lattice spanned by $\w_1, \ldots, \w_r$ is $E$.
 Since $E$ is an $\mathcal{O}$-module and $\w_i \in E$, the lattice spanned by $\w_1, \ldots, \w_r$ is trivially contained in $E$, so there only remains to prove the converse direction.
 Let us now assume by contradiction that there exists some $e_1, \ldots , e_r\in \C (\!(z)\!)$, not all of them belonging to $\C \llbracket z \rrbracket$, such that
 \[
  \e = \sum_{i=1}^r e_i \w_i \in E.
 \]
 Let $1 \leq i_0 \leq r$ be the largest integer $i$ such that $e_i \notin \C \llbracket z \rrbracket$.
 Then,
 \[
  \sum_{i=i_0 + 1}^r e_i \w_i \in E
 \]
 because $E$ is a lattice, $\w_i \in E$ and $e_i \in \C \llbracket z \rrbracket$ for $i_0 < i \leq r$.
 Since $E$ is closed under addition, we get
 \[
  \sum_{i=1}^{i_0} e_i \w_i \in E.
 \]
 Comparing with~\eqref{eq:vi}, we get that
 \begin{align*}
  e_{i_0} z^{n_i} \v'_i & = e_{i_0} \w_{i_0} - e_{i_0} b_{i_0-1,i_0} \w_{i_0-1} - \cdots  e_{i_0} b_{i_0-1,1} \w_1 \\
  & = \sum_{i=1}^{i_0} e_i \w_i - \sum_{i=1}^{i_0-1} (e_i -  e_{i_0} b_{i_0-1,i}) \w_i
 \end{align*}
  Now, the latter vector belongs to $E + \C (\!(z)\!) \w_{i-1} + \cdots + \C (\!(z)\!) \w_1$.
 However, as $e_i \notin \C \llbracket z \rrbracket$, we have $\operatorname{val}_z (e_{i_0} z^{n_i}) < n_i$, contradicting the minimal choice of $n_i$.
 \end{proof}

  Let us define $g\in \operatorname{Aut}(E)$ to be the gauge transformation defined by $\w_i  \cdot g = \e_i$, where $\w_i$ is the frame constructed in Proposition~\ref{lem:frame}.
  Notice that then $T_{\nabla\cdot g}$ is the monodromy transformation of $\nabla$ with respect to $(\w_1, \ldots , \w_r)$.

\begin{proposition}\label{prop:main}
  For this gauge transformation $g$, the connection matrix of $\nabla\cdot g$ reads as
 $$
   D \d\! z  = \sum_{n=-N}^{N'} D_n z^n \d\! z 
 $$
 for some $N'\in \Z_+$ (and the same value of $N$ as in~\eqref{eq:connection_series}), fulfilling
\begin{enumerate}
 \item $D_n\in\n$ for every $-N\leq n< -1$, \label{item:nilpotent}
 \item $D_n\in\b$ for every $-1 \leq n\leq N'$, \label{item:Borel}
 \item the monodromy transformation $T_{\nabla\cdot g}$ of $\nabla\cdot g$ leaves invariant the flag defined by $\e_1(1), \ldots ,\e_r(1)$,
 \item the $i$'th eigenvalue of $T_{\nabla\cdot g}$ is equal to 
 \[
  \exp (2\pi \sqrt{-1} \mu_i), 
 \]
 where $\mu_i = \operatorname{res}_0 d_{ii}$ and $d_{ii}$ denotes the $i$'th diagonal entry of $D$,  \label{item:monodromy}
 \item $D_{-N}$ is conjugate to $K_{-N}$.\label{item:conjugate}
\end{enumerate}
\end{proposition}

\begin{remark}
 For $N=1$, the statement is known as Levelt's normal form~\cite[Exercise~2.20]{Sabbah_Isomonodromic_deformations}. 
\end{remark}

\begin{proof}

We start by showing that $D$ is indeed a polynomial:

 \begin{lemma}
   For the frame $\w_i$ constructed in Proposition~\ref{lem:frame}, $\nabla (\w_i )$ is a linear combination of $\w_1, \ldots ,\w_i$ with coefficients in $\C [z^{\pm 1}]\d\! z$, and the diagonal entries have a pole of order at most $1$.
 \end{lemma}
 \begin{proof}[Proof of the Lemma]
 We apply induction on $i$. 
 The base case $i=1$ follows from 
 \begin{align*}
  \nabla (\w_1 ) & = \nabla ( z^{n_1} \v'_1 ) \\
  & = \frac{a_{11} + n_1}z \w_1 \d\! z. 
 \end{align*}
 For the induction step, applying the Leibniz formula to the lemma, we find 
 \[
  \nabla (\w_i ) = n_i \frac{\d\! z}{z} \w_i + z^{n_i} \nabla ( \v'_i ) +  \nabla (  b_{i-1,i} \w_{i-1} + \cdots +  b_{1i} \w_1). 
 \]
 By the induction hypothesis, the third term is a linear combination of $\w_1 , \ldots , \w_{i-1}$ with coefficients in $\C [z^{\pm 1}]\d\! z$. 
 By the assumption that $A$ preserves the flag, the second term reads as 
 \begin{equation}\label{eq:second_term}
  z^{n_i} \sum_{j<i} a_{ji} \v'_j \frac{\d\! z}{z} .  
 \end{equation}
 By induction, and using the upper-triangular nature of~\eqref{eq:vi}, this vector can be rewritten as 
 \begin{equation}\label{eq:crucial_step}
  (a_{ii} \w_i + c_{i-1,i} \w_{i-1} + \cdots +  c_{1i} \w_1 ) \frac{\d\! z}{z}  
 \end{equation}
 for some $c_{ji}\in \C [z^{\pm 1}]$, for $j<i$. 
 Summarizing, we see that 
 \begin{equation}\label{eq:recurrence} 
  \nabla (\w_i ) = (a_{ii} + n_i) \w_i \frac{\d\! z}{z} + d_{i-1,i} \w_{i-1} \d\! z + \cdots +  d_{1i} \w_1 \d\! z
  \end{equation}
 for some $d_{ji}\in \C [z^{\pm 1}]$, for $j<i$. 
 \end{proof}
 The lemma tells us that with respect to the system of vectors $(\w_1, \ldots , \w_r)$ in $E\otimes_{\C \llbracket z \rrbracket}\C (\!(z)\!)$, the connection matrix $D = (d_{ji})$ of $\nabla$ satisfies $d_{ji}\in \C [z^{\pm 1}]$ and $d_{ji}=0$ for $j>i$.
 Furthermore,~\eqref{eq:recurrence} shows that $\operatorname{val}_z (d_{ii})\geq -1$ and 
 \begin{equation}\label{eq:di}
  \operatorname{res}_0 d_{ii} = a_{ii} + n_i. 
 \end{equation}
 
 By~\eqref{eq:Id@1},~\eqref{eq:vi}, an upper triangular matrix $g$ links the bases
 \[
  (\v'_1(1), \ldots , \v'_r(1)) = (\v_1(1), \ldots , \v_r(1))
 \]
 and 
 \[
  (\w_1(1), \ldots , \w_r(1)).
 \] 
 In addition, its diagonal entries at $z=1$ are equal to $1^{n_i} = 1$. 
 By~(Cond3), $T$ preserves the flag defined by $(\w_1(1), \ldots , \w_r(1))$. 
 Moreover, by~\eqref{eq:expA} the eigenvalue of $T$ corresponding to $\w_i(1)$ is equal to 
 \[
 \exp (2\pi \sqrt{-1} a_{ii}), 
 \] 
 which is in turn equal to 
 \[
  \exp (2\pi \sqrt{-1} \operatorname{res}_0 d_{ii}) = \exp (2\pi \sqrt{-1} \mu_i) 
 \]
 by~\eqref{eq:di}.
 
 There only remains to show that $\operatorname{val}_z (d_{ji})\geq -N$ for every $j<i$. 
 Now, since $(\e_1, \ldots , \e_r)$ and $(\w_1, \ldots , \w_r)$ are both lattices of $E$, the gauge transformation $g\in \operatorname{Aut}(E)$ such that
 \begin{equation}\label{eq:wiei}
  (\w_1, \ldots ,\w_r) \cdot g = (\e_1, \ldots ,\e_r)
 \end{equation}
 is invertible (i.e., satisfies $\operatorname{val}_z (\det (g)) = 0$).
 By invertibility of $g$, we have $\operatorname{val}_z ( g^{-1}\d\! g) \geq 0$ and $\operatorname{val}_z (g^{-1} K g)= \operatorname{val}_z (K)=-N$. 
 We conclude using~\eqref{eq:gauge_action} that $\operatorname{val}_z (d_{ji})\geq -N$ for every $j<i$, and also that $g(0)^{-1} K_{-N} g(0) = D_{-N}$. 
\end{proof}

\begin{remark}
 The intuitive idea behind the proof of points~\eqref{item:nilpotent},~\eqref{item:Borel} of the theorem can be quite neatly summarized. 
 Namely, over the general point $\operatorname{Spec} \C (z)$ there exists a flag (i.e. a point in $\operatorname{GL} (r, \C )/B$) that is invariant under $T$: such a flag is given by a lattice $(\v'_1, \ldots , \v'_r)$, as in the proof. 
 Now, the flag variety $\operatorname{GL} (r, \C )/B$ is a proper scheme over $\C$. 
 By the valuative criterion for properness~\cite[Section~7.3]{EGAII}, the invariant flag then lifts to $\operatorname{Spec} \C [ z ]$. 
 Let $B_0\subset \operatorname{GL} (r, \C )$ be the flag obtained by base change to $\operatorname{Spec} \C [ z ]/(z)$ of the lift. 
 Invariance then means that the connection matrix must belong to $B_0$. 
 By general theory of regular singularities, the matrix coefficients $D_n$ for $n<-1$ must be nilpotent. 
 This argument gives a more conceptual proof of points~\eqref{item:nilpotent},~\eqref{item:Borel}. 
\end{remark}

Let us now give a partial converse to Proposition~\ref{prop:main}, even though it is fairly straightforward. 

\begin{proposition}\label{prop:converse}
 Let $\nabla$ be a connection on $V$ admitting connection matrix 
 $$
    \nabla = \d + \sum_{n=-N}^{N'} D_n z^n \d\! z 
 $$
 with respect to some $\O$-basis $(\v_1, \ldots ,\v_r)$ of $V$. Assume that 
 \begin{enumerate}
 \item $D_n\in\n$ for every $-N\leq n< -1$, 
 \item $D_n\in\b$ for every $-1 \leq n\leq N'$.
 \end{enumerate}
 Then $\nabla$ has regular singularity at $0$. 
\end{proposition}

\begin{proof}
 We are going to use the equivalent definition of regularity~\cite[(1.3)]{Malgrange}: $\nabla$ has regular singularity at $0$ if and only if over any sector $S_{\alpha,\beta}$ defined by $\alpha < \operatorname{arg} (z) < \beta$ in $\operatorname{Spec} \C [z, z^{-1}]$, $V$ admits a basis of $\nabla$-horizontal sections whose coefficients with respect to $(\v_1, \ldots ,\v_r)$ are at most of polynomial growth as $z\to 0$. 
 
 This property immediately follows by induction from the following result, whose proof in turn is a trivial application of separation of variables for the general solution of a linear ordinary differential equation. 
 \begin{lemma}
 Let 
 \[
  w'(z) = d(z) w(z) + f(z) 
 \]
 be a differential equation in $S_{\alpha,\beta}$ for some $d\in z^{-1} \C [z ]$. 
 Assume that $f$ is holomorphic in $S_{\alpha,\beta}$ and of at most polynomial growth as $z\to 0$. 
 Then the general solution $w$ is also holomorphic in $S_{\alpha,\beta}$ of at most polynomial growth as $z\to 0$. 
 \end{lemma}
\end{proof}

\section{Moduli space}\label{sec:moduli}

The results of Section~\ref{sec:normal} naturally lead us to consider the functor $NSQC_{\vec{c}}$ of non-semisimple quantum connections of rank $r$. 

\begin{theorem}\label{thm:main}
 Assume that $c_i \neq 1$ and $c_i \neq c_j$ for all $1\leq j \neq i \leq r$.
 The connected components of the coarse moduli space of the groupoid $NSQC_{\vec{c}}$ are parameterized by $\mathbb{Z}^r$. 
 Each connected component is isomorphic to $\n \oplus \n$. 
\end{theorem}

\begin{proof}
 Let us be given a triple $(E, \nabla,  (\v_1 (1), \ldots ,\v_r (1)))$ as in (Ob1)--(Ob3), satisfying conditions (Cond1)--(Cond3). 
 According to Proposition~\ref{prop:main} with $N=2$, there exist $g \in \operatorname{Aut}(E\otimes_{\mathcal{O}}  \mathcal{O}[z^{-1}])$ such that
 \[
  ( \d\! + K \d\! z ) \cdot g = \d\! + \sum_{n=-2}^{N'} D_n z^n \d\! z, 
 \]
 with $D_{-2}\in\n, D_{-1},\ldots , D_{N'}\in\b$. 
 By the assumption on the eigenvalues $c_i$, $T$ is regular semi-simple.
 It follows that the endomorphisms $A\in \gl (E_{|1})$ satisfying~\eqref{eq:expA} are parameterized by $\mathbb{Z}^r$.
 Moreover, for every choice of logarithm $A$ there exists a normal form as in Proposition~\ref{prop:main}, such that in addition the eigenvalues of $D_{-1} = \operatorname{res}_0 (D)$ agree with those of $A$.
 This is because one can pass from one such normal form to another one by using a diagonal gauge transformation with diagonal elements $z^{k_i}$, $k_i \in \mathbb{Z}$.
 Recall, on the other hand, that the morphisms in $NSQC_{\vec{c}}$ are given by $\operatorname{Aut}(E)$.
 Now, since gauge transformations by $\operatorname{Aut}(E)$ do not modify the eigenvalues of the residue, it follows that $NSQC_{\vec{c}}$ fibers over a $\mathbb{Z}^r$-torsor.
 From now on, we assume that we fix a base point $(\mu_1, \ldots , \mu_r)$ in this $\mathbb{Z}^r$-torsor.
 
 We now define the intermediate category $PCRS_{\vec{c}}$ of polynomial connections with regular singularities. 
 The objects of $PCRS_{\vec{c}}$ are $(E, \nabla)$ as in (Ob1)--(Ob2) satisfying (Cond1)--(Cond2), such that with respect to the standard basis $(\e_1, \ldots ,\e_r)$ we have 
 \[
  \nabla = \d\! + \sum_{n=-2}^{N'} D_n z^n \d\! z 
 \]
 \begin{itemize}
  \item $D_{-2}\in\n, D_{-1},\ldots , D_{N'}\in\b$ 
  \item the diagonal entries of $T_{\nabla}$ with respect to the standard  basis are $c_1, \ldots , c_r$. 
 \end{itemize}
 Notice that by Proposition~\ref{prop:converse}, $\nabla$ then necessarily has regular singularities.
 The morphisms of $PCRS_{\vec{c}}$ are given by holomorphic gauge transformations $\ell \in \operatorname{Aut}(E)$ satisfying:
 \begin{enumerate}
  \item $\ell$ takes values in the upper triangular Borel $B$,
  \item the eigenvalues of $\ell$ are holomorphic and invertible ($\ell_{ii} \in \mathcal{O}^{\times}$),
  \item $\ell(1) = \mbox{I}_r$.
 \end{enumerate}

 We define the functor 
 \begin{equation*}
  \Phi\colon NSQC_{\vec{c}} \to PCRS_{\vec{c}}
 \end{equation*}
 whose action on connections is given by 
 \[
   ( \d\! + K \d\! z ) \mapsto \d\! + \sum_{n=-2}^{N'} D_n z^n \d\! z, 
 \]
 and that forgets the framing of $NSQC_{\vec{c}}$. 
 The action of $\Phi$ on a morphism 
 \[
  h\colon (\nabla,  (\w_1 (1), \ldots ,\w_r (1)))\to (\tilde{\nabla},  (\tilde{\w}_1 (1), \ldots ,\tilde{\w}_r (1)))
 \]
 is then defined by the commutative diagram 
 \[
  \xymatrix{\nabla \ar[rr]^{\hspace{-1.5cm}g} \ar[d]^h && \nabla\cdot g = \d\! + \sum_{n=-2}^{N'} D_n z^n \d\! z \ar@{-->}[d]^{\Phi(h) = \tilde{g}\circ h \circ g^{-1}} \\
  \tilde{\nabla} = \nabla \cdot h \ar[rr]^{\hspace{-1.5cm}\tilde{g}} && \tilde{\nabla}\cdot \tilde{g} =  \d\! + \sum_{n=-2}^{N'} \tilde{D}_n z^n \d\! z
  }
 \]
 where $g,\tilde{g}$ are the gauge transformations produced by Proposition~\ref{lem:frame}.
 
 By Proposition~\ref{prop:converse}, every connection satisfying $D_{-2}\in\n, D_{-1},\ldots ,D_{N'}\in\b$ has regular singularity at $0$. 
 On the other hand, it is trivial that if $D$ satisfies the above conditions, then the monodromy around $0$ is upper triangular in the frame $(\e_1, \ldots , \e_r)$. 
 This shows that the functor $\Phi$ is essentially surjective. 
 
 We need to show that the action on morphisms is well-defined. 
 \begin{lemma}
  $\Phi(h)$ is an upper triangular holomorphic endomorphism with invertible holomorphic eigenvalues.
 \end{lemma}
 \begin{proof}[Proof of the Lemma]
 We are given~\eqref{eq:wiei} that for every $1\leq i \leq r$, 
 \[
  g(1)^{-1} \e_i = \w_i (1), \quad h (1) \w_i (1) = \tilde{\w}_i (1), \quad \tilde{g} (1) \tilde{\w}_i (1) = \e_i
 \]
 We infer that 
 \[
  \Phi(h) (1) \colon \e_i\mapsto \e_i, 
 \]
 in other words $\Phi(h) (1) = \mbox{I}_r$. 
 
 By construction, $g,h,\tilde{g}$ are holomorphic, so $\Phi(h)$ is holomorphic too, and its convergence radius is equal to the minimum of the convergence radii of $g,h,\tilde{g}$.
 
 We need to show that $\ell = \Phi(h)$ takes values in the standard Borel $B$.
 Let $E_i$ be the submodule of $E$ spanned by the basis vectors $(\e_1, \ldots , \e_i)$, so that we have the filtration
 \[
  E  = E_r \supset \cdots \supset E_1 \supset 0.
 \]
 Let us show that $\ell$ preserves this filtration.
 Indeed, assuming the contrary, there would exist $1 \leq i \leq r-1$ such that $\operatorname{Im}(\ell|_{E_i}) \not{\!\!\subseteq} E_i$.
 We can choose $i$ to be the maximal such value.
 Let $j\geq 1$ be the smallest integer such that
 \[
  \operatorname{Im}(\ell|_{E_i}) \subseteq E_{i+j}.
 \]
 Then, we have
 \[
  \ell (\v_i) = \ell_{i,i+j} \v_{i+j} + \sum_{1\leq k< i+j} \ell_{i,k} \v_k,
 \]
 with $\ell_{i,i+j} \in \mathcal{O} \setminus \{ 0 \}$.
 Notice that $\tilde{\nabla}$ is defined as the composition
 \[
  E_i \xrightarrow{\ell} E_{i+j} \xrightarrow{\nabla} E_{i+j} \xrightarrow{\ell^{-1}} E_{i+j},
 \]
 where the image of the last arrow $\ell^{-1}$ is indeed in $E_{i+j}$ by the maximal choice of $i$ and $j\geq 0$.
 It follows from~\eqref{eq:di} that
 \begin{equation}\label{eq:nablaellvi}
    \nabla (\ell (\v_i)) = \left( \operatorname{d} \! \ell_{i,i+j} + \frac{a_{i+j, i+j} + n_{i+j}}z \ell_{i,i+j} \right) \v_{i+j}
 \end{equation}
 up to lower order terms in $z$ and terms in $\v_k$ with $k< i+j$.
 Now, by the assumption $c_{i+j}\neq 1$, the residue of $\nabla$ has no integer eigenvalue, so we have
 \[
  a_{i+j, i+j} + n_{i+j} \notin \mathbb{Z}.
 \]
 This implies that the residue of the coefficient of $\v_{i+j}$ in~\eqref{eq:nablaellvi} does not vanish.
 In particular, the coefficient of $\v_{i+j}$ in~\eqref{eq:nablaellvi} does not vanish.
 Since $\ell$ is an isomorphism, it then follows that the coefficient of $\v_{i+j}$ in $\tilde{\nabla} (\v_i )$ does not vanish either.
 This is a contradiction, because $\tilde{\nabla}$ preserves the filtration $E_i$.

 Finally, let us show that the entries $\ell_{ii}$ are invertible holomorphic: the $i$'th diagonal entry of $\tilde{\nabla}$ is given by $d_{ii} + \operatorname{d} \! \log \ell_{ii}$.
 If
 \[
  \ell_{ii} = c z^{k} + O (z^{k + 1})
 \]
 for some $c\in \mathbb{C}\setminus \{ 0 \}$ and $k\in \mathbb{Z}$, then
 \[
  \operatorname{d} \! \log \ell_{ii} = c k z^{-1} \operatorname{d} \! z.
 \]
 Given that the residue of the $i$'th diagonal entry is fixed to be $\mu_i$ for both $\nabla$ and $\tilde{\nabla}$, we infer $k = 0$.

\end{proof}
 
 The functor $\Phi$ is fully faithful by definition. 
 It follows that the moduli space $NSQC_{\vec{c}}$ is equal to the coarse moduli space of $PCRS_{\vec{c}}$.
 The latter moduli space is the quotient of 
 \begin{equation}\label{eq:configuration_space}
    \mathcal{A} = \left\{ \d + \sum_{n=-2}^{N'} D_n z^n \d\! z \vert D_{-2}\in\n, D_{-1},\ldots ,D_{N'}\in\b \right\}
 \end{equation}
 by the group 
 \[
  \operatorname{Aut}_0(E) = \{ \ell \in \operatorname{Aut}(E) \, \vert \quad \ell \in B, \, \ell_{ii} \in \mathcal{O}^{\times}, \, \ell (1) = \mbox{I}_r \}.
 \]
 We will denote by $d_{ji}$ the $(ji)$-entry of $D$ for $j\leq i$. 
 Recall the notation $\mu_i = \operatorname{res}_0 d_{ii}$ from Proposition~\ref{prop:main}. 

 For every $1\leq i \leq r$, let us set 
 \[
  \ell_{ii} (z) = \exp \left( \int_1^z d_{ii} - \frac{\mu_i}{\zeta} \d\! \zeta \right) \in \O^{\times}.
 \]
 The choice of the definite integral ensures $\ell_{ii} (1) = 1$. 
 Then, the $i$'th diagonal entry of 
 \[
  \left( \d + \sum_{n=-2}^{N'} D_n z^n \right) \cdot \ell
 \]
 is equal to $\mu_i\frac{\d\! z}{z} $. 

 For every $1\leq j < i \leq r$, let us set 
 \begin{equation*}
    m(j,i) = \begin{cases}
               \mu_j - \mu_i \quad \mbox{if} \; \mu_j - \mu_i \in \N \\
              0 \quad \mbox{otherwise}.
             \end{cases}    
 \end{equation*}
 The theorem will follow from the 
 \begin{proposition}\label{prop:gauge}
  There exist a unique polynomial $\ell \in \operatorname{Aut}_0(E)$ and unique $\alpha_{ji}, \beta_{ji} \in \C$ such that for all $1\leq j < i \leq r$, the $(ji)$-entry of $\nabla\cdot \ell$ is equal to
 \[
  \alpha_{ji} z^{-2} + \beta_{ji} z^{m(j,i)-1}.
 \]
 Moreover, $\alpha_{ji}$ is equal to the $(ji)$-entry of $D_{-2}$.
 \end{proposition}
 \begin{proof}
 To find $\ell$, we apply induction on $j-i=k$, and exhibit $\ell$ as a product $\ell_1 \cdots \ell_{r-1}$.
 We assume that the $(j'i')$-entries with $j'-i'<k$ are of the stated form, and we look for $\ell$ such that $\ell_{j'i'} = \delta_{j'i'}$ for $j'-i'<k$.
 Gauge transformation by such an $\ell$ does not modify connection matrix entries with $j'-i'<k$.
 The action on $d_{ji}$ is
 \begin{align*}
  d_{ji} & \mapsto d_{ji} + z^{-1} (\mu_i - \mu_j ) \ell_{ji} + \partial_z \ell_{ji}\\
  & = d_{ji} + z^{\mu_j - \mu_i} \partial_z \left( z^{\mu_i - \mu_j} \ell_{ji} \right).
 \end{align*}
 For $\mu\in\C$, let us define the linear map
 \[
  L_{\mu} \colon \{ P\in\C [z]\colon P(1) = 0 \} \to z^{-2} \C [z]
 \]
 by
 \[
  P\mapsto z^{\mu} \partial_z \left( z^{-\mu} P \right).
 \]
 The action of $L_{\mu}$ on the standard basis of Laurent polynomials is given by
 \begin{equation}\label{eq:Lmu}
  L_{\mu}\colon z^n \mapsto (n-{\mu}) z^{n-1}.
 \end{equation}
 In particular, $z^{-2}\notin \operatorname{Im}(L_{\mu})$ regardless of the value of $\mu$.
 \begin{lemma}
 \begin{enumerate}
  \item  For ${\mu}\notin\N$, $\operatorname{coker}(L_{\mu})$ is generated by $\C \cdot z^{-2}, \C \cdot z^{-1}$.
  \item For ${\mu}\in\N$, $\operatorname{coker}(L_{\mu})$ is generated by $\C \cdot z^{-2}, \C \cdot z^{{\mu}-1}$.
 \end{enumerate}
 \end{lemma}
  \begin{proof}[Proof of the Lemma]
  \begin{enumerate}
   \item For ${\mu}\notin\N$, let us be given any $Q\in z^{-1} \C [z]$: 
   \[
    Q(z) = q_{-1} z^{-1} + \cdots + q_{M-1} z^{M-1}. 
   \]
   Then, $m-\mu\neq 0$ for all $m\in\mathbb{N}$, and we see from~\eqref{eq:Lmu} that 
   \[
    L_{\mu}\colon -\mu^{-1} q_{-1} + \cdots + (M-\mu)^{-1} q_{M-1} z^M \mapsto Q(z).
   \]
   The only obstruction for $Q\in\operatorname{Im} (L_{\mu})$ is a linear relation on the coefficients $q_{-1}, \ldots, q_{M-1}$, expressing the vanishing at $z=1$ of the polynomial on the left hand side.
   This means that $\dim(\operatorname{coker}(L_{\mu})) = 2$. 
   In particular, $z^{-1}$ represents a nontrivial class in $\operatorname{coker}(L_{\mu})$, and together with $z^{-2}$ they form a basis of it. 
   \item For $\mu\in\N$, a basis for the source of $L_{\mu}$ is provided by $z^n - z^{\mu}$ for $0\leq n \neq {\mu}$, and 
 \[
  L_{\mu}\colon z^n - z^{\mu} \mapsto (n-{\mu}) z^{n-1}.
 \]
 Therefore, the only basis element of the target that is not in the image of $L_{\mu}$ is $z^{\mu-1}$.
  \end{enumerate}
\end{proof}
 Applying $L_{\mu_j - \mu_i}$, there exists a unique $\ell_{ji}(z)$ such that 
 \[
  d_{ji} (z) + L_{\mu_j - \mu_i} (\ell_{ji}(z)) = \alpha_{ji} z^{-2} + \beta_{ji} z^{m(j,i)-1}
 \]
 and $\ell_{ji}(1) = 0$. This finishes the proof of Proposition~\ref{prop:gauge}.
\end{proof}
We find (see~\eqref{eq:configuration_space})
\[
 \mathcal{A} / \operatorname{Aut}_0(E) = \{ \alpha_{ji} z^{-2} + \beta_{ji} z^{m(j,i)-1} \; \vert \quad \alpha_{ji}, \beta_{ji} \in \C \} \cong  \n \oplus \n. 
\]
This finishes the proof of Theorem~\ref{thm:main}.
\end{proof}

\section{Further aspects}

\subsection{Symmetric group action} 
Assuming the non-resonance condition $c_i \neq c_j$, for any object $(E, \nabla,  (\v_1 (1), \ldots ,\v_r (1)))\in\operatorname{Ob}(NSQC_{\vec{c}})$ there exists a constant gauge transformation bringing $T$ to diagonal form. 
Assume that we have achieved this. 
We may then rearrange the eigenvalues $c_i$ of $T$ in any possible order. 
Namely, for any permutation $\sigma\in\mathfrak{S}_r$, we may consider the vector $\sigma\cdot \vec{c}$ defined by 
\[
 (\sigma\cdot \vec{c})_i =  \vec{c}_{\sigma^{-1}(i)}. 
\]
Then, 
\[
 (E, \nabla,  (\v_{\sigma^{-1}(1)} (1), \ldots ,\v_{\sigma^{-1}(r)} (1)))\in\operatorname{Ob}(NSQC_{\sigma\cdot\vec{c}}). 
\]
In this way, we get a (nonlinear) action of $\mathfrak{S}_r$ on the vector space $\n \oplus \n$. 
Similarly, supposing that the construction of Theorem~\ref{thm:main} goes through in family, it would be possible to define an action of the braid group $\pi_1 (H, \vec{c})$ on $\n \oplus \n$, where $H\subset \operatorname{GL} (r, \C )$ is the maximal torus. 
It would be worth to pursue further the study of these actions.

\subsection{Relationship to tame parahoric bundles}

Our main result and its proof bear some resemblance to the  Riemann--Hilbert correspondence for tame parahoric (also called logahoric) bundles~\cite[Theorem~D]{Boalch_logahoric} (and hence, to the Grothendieck--Springer resolution).
Indeed, the latter theorem states an equivalence of groupoids between some regular connections which may admit poles of arbitrarily high order (with suitable loop group element actions as morphisms) on one side, and pairs consisiting of a conjugate $P$ of a fixed parabolic subgroup $P_{\phi}\subset G$ together with an element $T\in P$ (with conjugation as morphisms) on the other side. 
Indeed, we also have regular connections with higher order poles on one side, and a parabolic $P$ of a given type (namely, the subgroup of elements preserving some full flag $(\v_1 (1), \ldots ,\v_r (1))$) and a transformation $T\in P$ (the monodromy) on the other side. 
However, there is a major difference too: the valuation that a coefficient of a parahoric bundle is allowed to have depends on the integer part of the corresponding weight of a fixed covector. 
Up to applying a Weyl group element, this condition allows poles of different orders for the entries above the diagonal, and even more crucially the entries below the diagonal are not necessarily zero, only vanishing to a certain order. 
This relationship would be worth to be explored further.

\end{document}